\documentclass[11pt,reqno]{amsart}
\usepackage[english]{babel}
\usepackage{amssymb}
\usepackage{amsmath}
\usepackage{epsfig}
\usepackage{graphicx}
\usepackage{dsfont}
\usepackage{caption}
\usepackage{subcaption}

\theoremstyle{plain}
\newtheorem{theorem}{Theorem}
\newtheorem{lemma}{Lemma}
  \def\Im{\mathop{\rm Im}\nolimits}

\title{Value range of solutions to the chordal Loewner equation with restriction on the driving function}
\author{A. V. Zherdev}
\thanks{This work was supported by the Russian Science Foundation (project no.~17-11-01229)}

\begin{document}
\maketitle \markright{Value range of solutions to the chordal Loewner equation}

\begin{abstract}
We consider a value range $\{g(i,T)\}$ of solutions to the chordal Loewner equation with the restriction $|\lambda(t)| \le c$ on the driving function. We use reachable set methods and the Pontryagin maximum principle.
\end{abstract}

{\bf 1. Introduction. } 
Problems of finding value ranges $\{f(z_0)\}$ are typical for the geometric function theory. Here functions $f$ are taken from some class of analytic functions and $z_0$ is a fixed point in the domain of functions from that class.

A number of problems of this kind have been solved for classes of analytic functions defined in the unit disk $\mathds{D}=\{z: |z|<1\}$. Rogosinski \cite{Rog34} gave a description of the value range $\{f(z_0)\}$ for the class of all analytic functions mapping the unit disk $\mathds{D}$ into itself, $f(0) = 0,\, f'(0) \ge 0$. Grunsky \cite{Grunsky32} described the value range $\{\log(f(z_0)/z_0) : f \in \mathcal{S}\},\, z_0 \in \mathds{D}$ within the class $\mathcal{S}$ of univalent analytic functions $f$ in $\mathds{D}$, $f(0) = 0,\, f'(0) = 1$. Goryainov and Gutlyanski \cite{GG76} extended this result by describing the set $\{\log(f(z_0)/z_0) : f \in \mathcal{S}_M\}$ for the subclass $\mathcal{S}_M = \{f \in \mathcal{S} : |f| \le M \}$ of bounded functions.

Roth and Schleissinger \cite{RS14} described the value range $\{f(z_0)\}$ for all analytic univalent functions $f:\mathds{D}\to\mathds{D}$, $f(0) = 0,\, f'(0) > 0$, that is, they obtained an analogue of Rogosinski's result for univalent functions. In the same article they found a description of the set $\{g(z_0)\}$ within the class of all univalent analytic functions $g:\mathds{H}\to\mathds{H}$, mapping the upper half-plane $\mathds{H}=\{z: \Im z > 0\}$ into itself and normalized $g(z) = z + cz^{-1} + O(|z|^{-2}),\, z\to \infty$. Value ranges for some classes of analytic univalent functions defined in $\mathds{D}$ were described in \cite{KS16, KS17}.

Denote $\mathcal{H}(T),\, T > 0$ the class of all analytic univalent functions $g:\mathds{H} \backslash K \to \mathds{H}$, normalized near infinity as
$$g(z) = z + \frac{2T}{z} + O(|z|^{-2}).$$ Here $K \subset \mathds{H}$ is a so-called hull, which means that $K=\mathds{H}\cap\overline{K}$
and $\mathds{H}\backslash K$ is simply connected. Solutions of the chordal Loewner differential equation
\begin{equation}
\label{Zher_eq1}\frac{dg(z, t)}{dt} = \frac{2}{g(z,t) - \lambda(t)},\, g(z,0) = z,\, t \geq 0,
\end{equation}
where $\lambda(t)$ is a real-valued continuous function, form a dense subclass of $\mathcal{H}(T)$. We call $\lambda(t)$ the driving function of the chordal Loewner equation \eqref{Zher_eq1}. Thus, the problem of finding the value range $\{g(z_0): g \in \mathcal{H}(T)\},\ z_0 \in \mathds{H}$,
is equivalent to describing the set $\{g(z_0, T)\}$ of attainability of the equation \eqref{Zher_eq1}. Without loss of generality we can put $z_0 = i$. The set
$$D(T) = \{g(i, T): g \text{ solution \eqref{Zher_eq1}}\}$$
has been described by Prokhorov and Samsonova in \cite{ProhorovSamsonova15} using the Pontryagin maximum principle. They proved the following theorems.

\begin{theorem}\cite{ProhorovSamsonova15} The domain $D(T),\, 0 < T \le \tfrac{1}{4}$, is bounded by two curves $l_1$ and $l_2$ connecting the points $i$ and $i\sqrt{1-4T}$. 
The curve $l_1$ in the complex $(x,y)$-plane is parameterized by the equations
$$x(T) = \frac{C_{0}^{2}(\varphi, T)(4T-1)+(1-\sin\varphi)^2}{2C_{0}(\varphi, T)\cos\varphi},\,
y(T) = \frac{1-\sin\varphi}{C_{0}(\varphi, T)},\, -\frac{\pi}{2} < \varphi < \frac{\pi}{2},$$
where $C_{0}(\varphi, T)$ is the unique root of the equation
$$2\cos^{2}\varphi\log(1-\sin\varphi) + (1-\sin\varphi)^2 = 2\cos^2\varphi\log C + C^2(1 - 4T).$$
The curve $l_2$ is symmetric to $l_1$ with respect to the imaginary axis.
\end{theorem}

\begin{theorem}\cite{ProhorovSamsonova15} The domain $D(T),\, T > \tfrac{1}{4}$, is bounded by two curves $l_1 = l_{11} \cup l_{12}$ and $l_2$
which is symmetric to $l_1$ with respect to the imaginary axis, $l_1$, $l_2$ have the mutual endpoint $i$. 
The curve $l_{11}$ in the complex $(x,y)$-plane is parameterized by the equations
$$x(T) = \frac{C_{0}^{2}(\varphi, T)(4T-1)+(1-\sin\varphi)^2}{2C_{0}(\varphi, T)\cos\varphi},\,
y(T) = \frac{1-\sin\varphi}{C_{0}(\varphi, T)},$$
where $\varphi_0(T) < \varphi < \frac{\pi}{2}$. The curve $l_{12}$ is parameterized by the equations
$$x(T) = \frac{C_{00}^{2}(\varphi, T)(4T-1)+(1-\sin\varphi)^2}{2C_{00}(\varphi, T)\cos\varphi},\,
y(T) = \frac{1-\sin\varphi}{C_{00}(\varphi, T)},$$
$\varphi_0(T) < \varphi < \frac{\pi}{2}$. Here $C_{0}(\varphi, T) > 0$ and $C_{00}(\varphi, T) > 0$ are the minimal and maximal roots of the equation
$$2\cos^{2}\varphi\log(1-\sin\varphi) + (1-\sin\varphi)^2 = 2\cos^2\varphi\log C + C^2(1 - 4T),$$
respectively, $\varphi_0(T) \in (-\frac{\pi}{2}, \frac{\pi}{2})$ is the unique solution of the equation
$$\log\frac{1-\sin\varphi}{1+\sin\varphi} + \frac{1-\sin\varphi}{1+\sin\varphi} + 1 = -\log(4T-1).$$
\end{theorem}

Continuing this research we consider a problem of describing the value range
$$D_c(T) = \{g(i, T): g \text{ solution \eqref{Zher_eq1}},\ |\lambda(t)| \le c\},$$
that is, we added the restriction $|\lambda(t)| \le c$ on the driving function, which is piecewise continuous on $\mathds{R}$. We use the Pontryagin maximum principle as the main tool of the research. See \cite{Pro90, Pro93} for reachable set methods developed for the radial Loewner differential equation.

{\bf 2. Preliminary Statements.}
 Due to a well known property of the Loewner equation \eqref{Zher_eq1} (see, for example, \cite{Kag04}) and symmetry of the restriction $|\lambda(t)| \le c$, the domain $D_c(T)$ is symmetric with respect to the imaginary axis. Therefore, we can consider only the right half $(x \ge 0)$ of the domain.

Putting $z=i$ in the Loewner differential equation \eqref{Zher_eq1} and splitting the result equation into real and imaginary parts we obtain the system
of ordinary differential equations

\begin{equation}
\label{Zher_eq2}
  \begin{aligned}
     \frac{dx}{dt} &= \frac{2(x-\lambda)}{(x-\lambda)^2 + y^2},\,&x(0)=0,\\
     \frac{dy}{dt} &= -\frac{2y}{(x-\lambda)^2 + y^2},\,&y(0)=1.
  \end{aligned}
\end{equation}
Following the Pontryagin maximum principle formalism we introduce an adjoint vector $\Psi(t)=(\Psi_1(t), \Psi_2(t)) \ne 0$ and the Hamilton function
\begin{equation}
\label{Zher_eq3} H(x,y,\Psi_1, \Psi_2, \lambda) = \frac{2(x-\lambda)\Psi_1 - 2y\Psi_2}{(x-\lambda)^2 + y^2}.
\end{equation}
The adjoint vector satisfies the system
\begin{equation}
\begin{aligned}
\label{Zher_eq4} 
\frac{d\Psi_1}{dt} &= -\frac{\partial H}{\partial x} = 
\frac{2}{((x-\lambda)^2 + y^2)^2}[((x-\lambda)^2 - y^2)\Psi_1 - 2(x-\lambda)y\Psi_2],\\
\frac{d\Psi_2}{dt} &= -\frac{\partial H}{\partial y} = 
\frac{2}{((x-\lambda)^2 + y^2)^2}[(2(x-\lambda)y\Psi_1 + ((x-\lambda)^2 - y^2)\Psi_2].
\end{aligned}
\end{equation}
The domain $D_c(T)$ is a set of attainability for the phase system \eqref{Zher_eq2} at $t=T$.
A boundary point $A=x_A(T) + iy_A(T)$ of $D_c(T)$ is delivered by the solution $(x_A(t), y_A(t))$ of the Hamiltonian system \eqref{Zher_eq2}-\eqref{Zher_eq4} with the driving
function $\lambda_A(t)$ satisfying the Pontryagin maximum principle

\begin{equation*}
\begin{aligned}
\max\limits_{\lambda \in [-c, c]}H(x_A(t), y_A(t), &\Psi_1^A(t), \Psi_2^A(t), \lambda) \\
&= H(x_A(t), y_A(t), \Psi_1^A(t), \Psi_2^A(t),\lambda_A(t))
\end{aligned}
\end{equation*}
at continuity points of $\lambda_A(t)$. Note that
$$\lim_{\lambda \to \infty}H(x, y, \Psi_1, \Psi_2, \lambda) = \lim_{\lambda \to -\infty}H(x, y, \Psi_1, \Psi_2, \lambda) = 0$$
for any fixed values of $x, y, \Psi_1, \Psi_2$. Therefore, the maximum of $H$ is attained at zeros of the derivative of $H$ with respect to $\lambda$
$$\frac{\partial H(x, y, \Psi_1, \Psi_2, \lambda)}{\partial\lambda} = 2\frac{((x-\lambda)^2-y^2)\Psi_1 - 2(x-\lambda)y\Psi_2}{(x-\lambda)^2 + y^2}.$$
It is not difficult to show that $H$ has only one local maximum on $\mathds{R}$ for any fixed values of $x, y, \Psi_1, \Psi_2$ at
\begin{equation}
\label{Zher_eq5}
\lambda_0 = x-\frac{y\Psi_1}{\sqrt{\Psi_{1}^{2} + \Psi_{2}^{2}} - \Psi_2}.
\end{equation}
Therefore, $H$ attains its maximum on the interval $[-c, c]$ either at $\lambda_0$ if $\lambda_0 \in [-c, c]$, or at one of the endpoints of the
interval, otherwise.

We formulate the following lemma providing differential equations for the phase trajectory $(x(t), y(t))$ in the case when $\lambda_0 \in [-c, c]$.

\begin{lemma}
Let $\lambda(t)$ maximize the Hamilton function \eqref{Zher_eq3} on $\mathds{R}$ for $t \in [t_1, t_2] \subset [0, T]$, that is,
$$\max\limits_{\lambda \in \mathds{R}}H(x(t), y(t), \Psi_1(t), \Psi_2(t), \lambda) = H(x(t), y(t), \Psi_1(t), \Psi_2(t),\lambda(t)),$$
where $(x(t), y(t))$ is a solution of the phase system \eqref{Zher_eq2} and $(\Psi_1(t), \Psi_2(t))$ is a solution of the adjoint system \eqref{Zher_eq4}. Then

(a) $\frac{\Psi_1 y}{\sqrt{\Psi_1^2 + \Psi_2^2}-\Psi_2} \equiv const = p,\, t\in [t_1,\, t_2],$

(b) $\lambda(t) = x(t) - p,\, t\in [t_1,\, t_2],$

(c) the phase trajectory $(x(t), y(t))$ satisfies the following differential equations
\begin{equation}
\label{Zher_eq6}
\frac{dy}{dt} = -\frac{2 y}{p^2 + y^2},
\end{equation}
\begin{equation}
\label{Zher_eq7}
\frac{dx}{dy} = -\frac{p}{y}.
\end{equation}
\end{lemma}

\begin{proof}
Since $\lambda(t)$ maximizes $H$ on $\mathds{R}$, it satisfies \eqref{Zher_eq5} for $t \in [t_1,\, t_2]$.
By substituting \eqref{Zher_eq5} into \eqref{Zher_eq2}-\eqref{Zher_eq4} we obtain
\begin{equation}
\label{Zher_eq8} \frac{dx}{dt} = \frac{\Psi_1}{y\sqrt{\Psi_1^2 + \Psi_2^2}},
\end{equation}
\begin{equation}
\label{Zher_eq9} \frac{dy}{dt} = -\frac{\sqrt{\Psi_1^2 + \Psi_2^2}-\Psi_2}{y\sqrt{\Psi_1^2 + \Psi_2^2}}.
\end{equation}
\begin{equation}
\label{Zher_eq10} \frac{d\Psi_1}{dt} = 0,\, \frac{d\Psi_2}{dt} = \frac{\sqrt{\Psi_1^2 + \Psi_2^2}-\Psi_2}{y^2}.
\end{equation}
\begin{equation*}
H(x,y,\Psi_1, \Psi_2, \lambda_0) = \frac{\sqrt{\Psi_1^2 + \Psi_2^2}-\Psi_2}{y}.
\end{equation*}
In view of \eqref{Zher_eq10} we have 
$$\Psi_1(t) \equiv const = c_1,\,t \in [t_1, t_2].$$
Due to a well-known property of the Hamilton function we have
\begin{equation}
\label{Zher_eq11}
H(x, y, \Psi_1, \Psi_2, \lambda) = \frac{\sqrt{\Psi_1^2 + \Psi_2^2}-\Psi_2}{y} \equiv const = c_2,\, t \in [t_1, t_2].
\end{equation}

We put $p = \frac{c_1}{c_2}$. Thus, we proved statements (a) and (b). Using \eqref{Zher_eq11} we can rewrite \eqref{Zher_eq9} as \eqref{Zher_eq6}. By dividing \eqref{Zher_eq8} by \eqref{Zher_eq9} we obtain the differential equation \eqref{Zher_eq7}.
\end{proof}

We note that equations of the Hamiltonian system \eqref{Zher_eq2}-\eqref{Zher_eq4} are invariant under changing the sign of $\Psi_1(t)$, $x(t)$ and $\lambda(t)$ to the opposite. Thus, flipping the sign of $\Psi_1(t)$ (and, due to the statement (a) of Lemma 1, equally of $p$) has the effect of reflecting the phase trajectory $(x(t), y(t))$ in the imaginary axis. Therefore, we can restrict ourselves to the case of $p \ge 0$. We will see that this choice will lead us to the right half of the boundary of $D_c(T)$.

In the case of no restrictions on the driving function $\lambda(t)$ we have $c=\infty$ and the condition $\lambda_0 \in [-c, c]$ always holds. This allows us to deduce from Lemma 1 a description of the boundary of $D(T)$ in the Cartesian coordinates $(X,Y)$.

\begin{theorem}
The boundary of the domain $D(T),\, T > 0$ is given by the equation
\begin{equation}
\label{Zher_eq12}
2X^2 = \log Y(1-4T-Y^2).
\end{equation}
\end{theorem}

\begin{proof}
Since conditions of Lemma 1 are satisfied on the whole interval $[0, T]$, we can integrate the equations \eqref{Zher_eq6} and \eqref{Zher_eq7} over this interval with the conditions $x(0)=0,\, y(0)=1,\, x(T)=X,\, y(T)=Y$. We obtain
\begin{equation}
\label{Zher_eq13}
2p^2\log Y + Y^2 = 1-4T,\, X = - p \log Y.
\end{equation}
Finally, multiplying the first of these equations to $\log Y$ and using the second we obtain \eqref{Zher_eq12}.
\end{proof}

\begin{figure}
\centering
    \begin{subfigure}[h]{0.495\textwidth}
        \epsfig{file=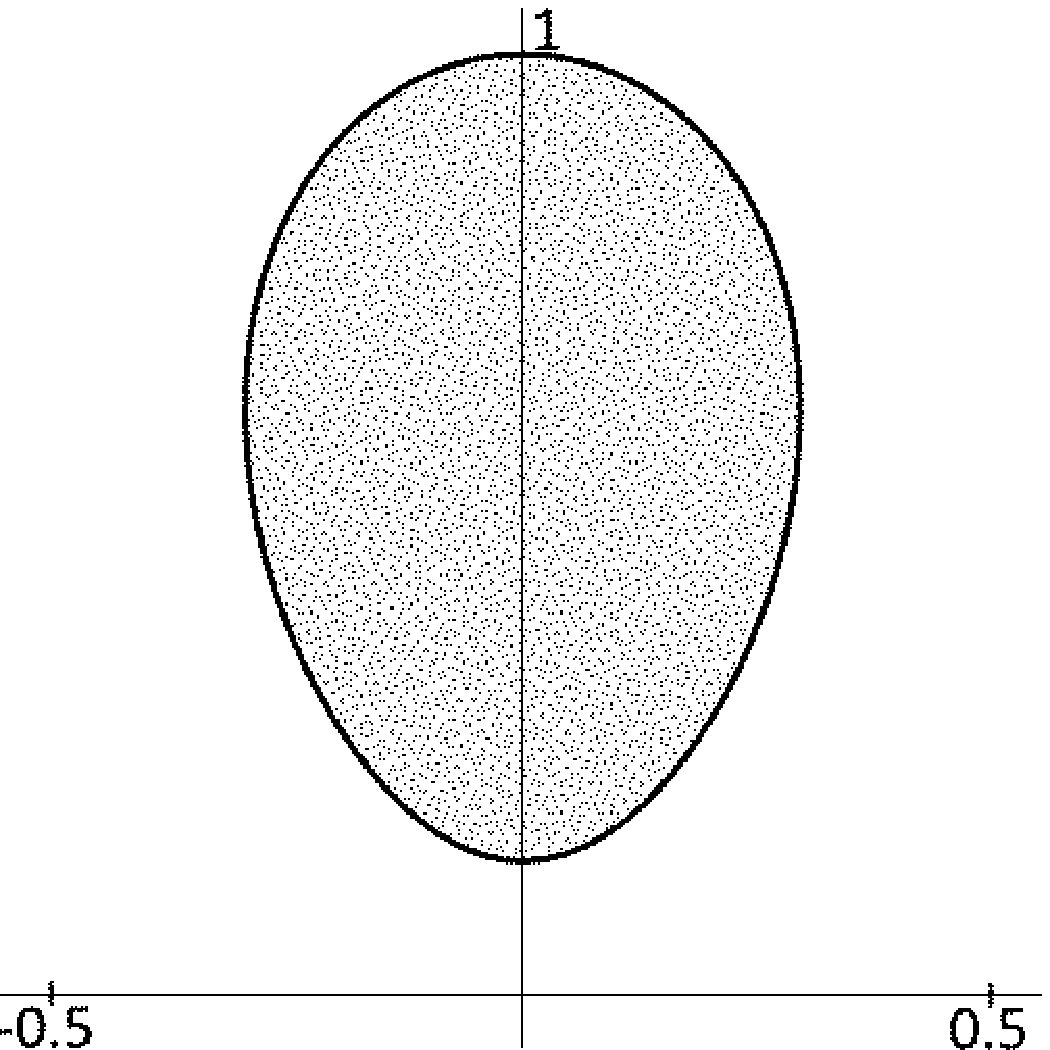, height=1.7in, width=1.7in}
        \caption{T=0.245}
    \end{subfigure}
        \centering
        \hspace{-12mm}
        \begin{subfigure}[h]{0.495\textwidth}
        \epsfig{file=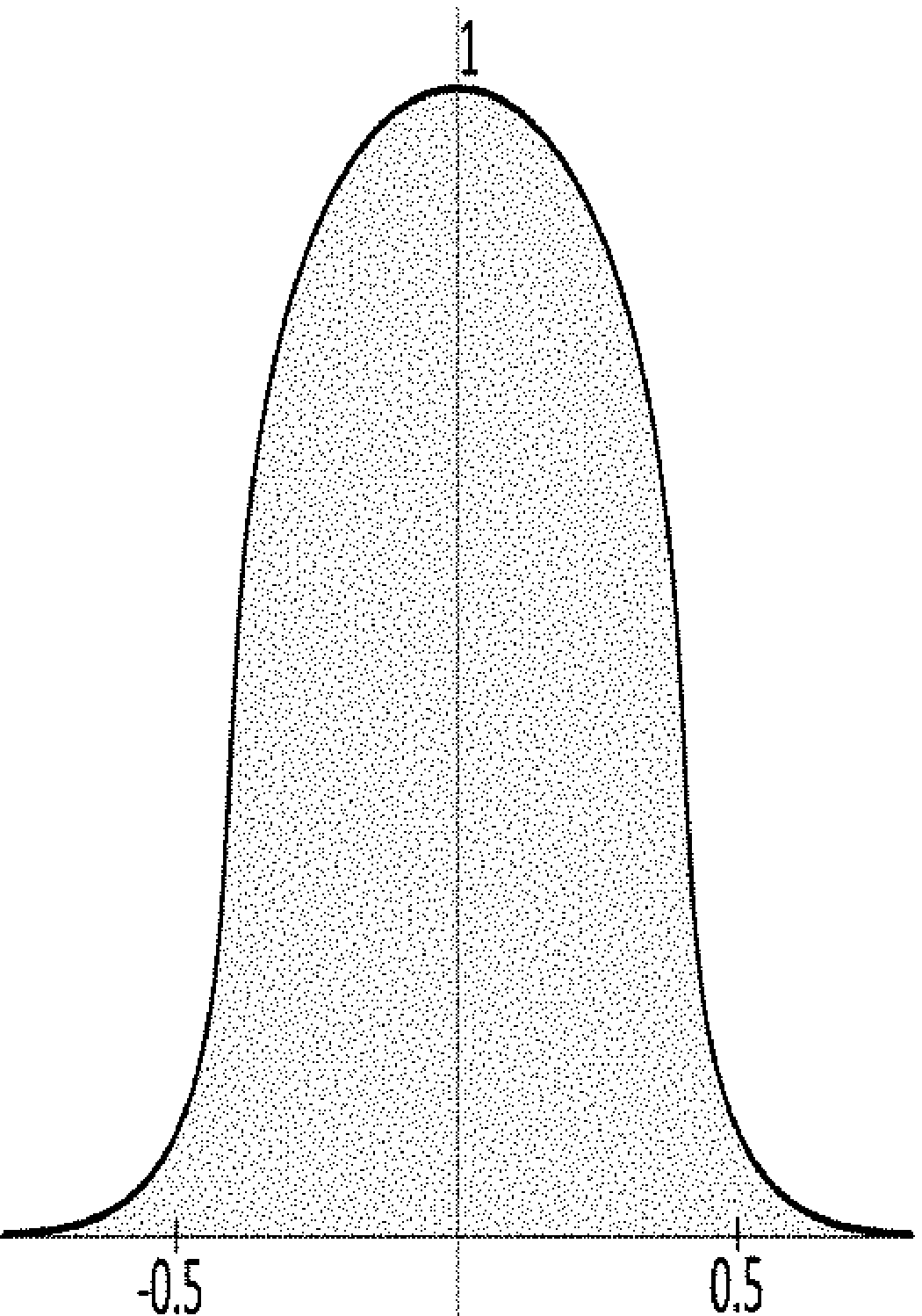, height=1.7in, width=2.43in}
        \caption{T=0.3}
    \end{subfigure}
\caption{Value ranges $D(T)$}
\label{Zher_fig1}
\end{figure}

It is easy to see that there are two essentially different cases. In the case of $T \le \frac{1}{4}$ the set $D(T)$ is a bounded domain with its boundary crossing the imaginary axis at $y=\sqrt{1-4T},\, y=1$. This case corresponds to Theorem 1. If $T > \frac{1}{4}$ the set $D(T)$ is unbounded and its boundary includes the real axis, this case corresponds to Theorem 2. Starting at this point, we only consider the case of $T \le \tfrac{1}{4}$.

 Note that the boundary point $(0, \sqrt{1-4T})$ of $D(T)$ is delivered
by the driving function $\lambda(t)\equiv 0$. Therefore, it also belongs to the boundary of $D_c(T)$.
It is a reasonable assumption that all points of some arc on $\partial D(T)$ near $(0, \sqrt{1-4T})$ are delivered by driving functions
with ranges within the interval $[-c, c]$, and since this arc belongs to $\partial D_c(T)$.
A precise statement is given by the following lemma.

\begin{lemma}
A segment of the boundary $\partial D_c(T)$ is given by \eqref{Zher_eq12}, $Y \in [1-4T, Y_0]$, $Y_0$ is the unique solution of one of the equations
\begin{equation}
\label{Zher_eq14}
2c^2\log Y + Y^2 = 1 - 4T,\, c^2 \ge T - \frac{1-e^{-4}}{4},
\end{equation}
\begin{equation}
\label{Zher_eq15}
\frac{2c^2\log Y}{(1 + \log Y)^2} + Y^2 = 1 - 4T,\, c^2 \le T - \frac{1-e^{-4}}{4}.
\end{equation}
\end{lemma}

Note that if $c^2 = T - \tfrac{1-e^{-4}}{4}$ both equations \eqref{Zher_eq14}, \eqref{Zher_eq15} have the same root $Y_0 = e^{-2}$.

\begin{proof}
Consider a point on the boundary $\partial D(T)$. Let $\lambda_0(t)$ denote the driving function delivering this point.
By Lemma 1 we have $\lambda_0(t) = x(t) - p$. Since  $p > 0$, we can see from \eqref{Zher_eq8} that $x(t)$ and, hence, $\lambda_0(t)$ are increasing functions.

A boundary point of $D(T)$ belongs to the boundary of $D_c(T)$ if it is delivered by a driving function with the range within $[-c, c]$.
Since $\lambda_0(t)$ is increasing this condition is equivalent to inequalities 
\begin{equation}
\label{Zher_eq16}
\lambda_0(0) \ge -c,\, \lambda_0(T) \le c.
\end{equation}
We note that $\lambda_0(0) = -p,\, \lambda_0(T) = X - p$. Equations \eqref{Zher_eq13} allow us to express $X$ and $p$ through $Y$. Substituting it into \eqref{Zher_eq16} and squaring the result we obtain
$$
\frac{1-4T-y^2}{2 \log y} \le c^2,\, \frac{1-4T-y^2}{2 \log y}(1 + \log y)^2 \le c^2.
$$
We need to find the greatest value $Y_0$ of $Y$ satisfying both conditions. Define the following functions of $Y$ for $Y \in [\sqrt{1-4T}, 1]$
$$f_1(Y) = \frac{1-4T-Y^2}{2 \log Y},\, f_2(Y)=\frac{1-4T-Y^2}{2 \log Y}(1 + \log Y)^2.$$
It is easy to see that $f_1(Y) \ge f_2(Y)$ for $Y \in [e^{-2}, 1]$, in particular, it is always true if
$\sqrt{1-4T} \ge e^{-2}$ or, which is the same, $T - \tfrac{1-e^{-4}}{4} \le 0$. Therefore, in this case
$Y_0$ is the solution of $f_1(Y) = c^2$ which is equivalent
to \eqref{Zher_eq14} and it remains to prove the case of $\sqrt{1-4T} < e^{-2}$.

We have $f_1(Y) \le f_2(Y)$ and the equality sign holds only at $Y=e^{-2}$. Hence, we need to check if $f_2$ attains the 
value $c^2$ within the interval $[\sqrt{1-4T}, e^{-2}]$. The derivative
$$f_2'(Y) = \frac{1 + \log Y}{2 (\log Y)^2}(-2Y(1+\log Y)\log Y + \frac{\log Y - 1}{Y}(1-4T-Y^2))$$
vanishes at $Y = e^{-1}$ and at roots of the equation
$$2\frac{\log Y + 1}{\log Y - 1} = \frac{1-4T}{Y^2} - 1.$$
The left-hand side of the equation is an increasing function of $Y$ on $[\sqrt{1-4T}, e^{-2}]$ and takes the value
$-\tfrac{4}{3}$ at $Y=e^{-2}$, while the right-hand side is decreasing on $[\sqrt{1-4T}, e^{-2}]$ and takes the value
$\tfrac{1-4T}{e^{-4}} - 1 > -1$ at $Y=e^{-2}$. Therefore, the derivative $f_2'$ does not vanish on the interval $[\sqrt{1-4T}, e^{-2}]$.
Since $f_2'(\sqrt{1-4T}) > 0$, $f_2$ increases on $[\sqrt{1-4T}, e^{-2}]$. Therefore, $f_2$ attains its maximum at $Y=e^{-2}$.
Hence, $Y_0$ is the solution of $f_2(Y) = c^2$ if the inequality $f_2(e^{-2}) > c^2$ holds. We note that the last inequality
gives $c^2 > T - \frac{1-e^{-4}}{4}$ to complete the proof.
\end{proof}

If $\lambda(t)\equiv\pm c$, the phase system \eqref{Zher_eq2} can be integrated directly. We need the following properties of its solutions stated by the lemma below.

\begin{lemma}
If trajectory $(x(t),y(t))$ satisfies
$$\frac{dx}{dt} = \frac{2(x-a)}{(x-a)^2 + y^2},\, \frac{dy}{dt} = -\frac{2y}{(x-a)^2 + y^2},$$
where $a$ is a real number, then the following quantities are constant
\begin{equation*}
(x-a)y,\, (x-a)^2-y^2-4t.
\end{equation*}
\end{lemma}
\begin{proof}
The statement can be proved by direct integration of the system.
\end{proof}

{\bf 3. Main Theorem.} Now we are ready to prove the following theorem describing the value range $D_c(T)$ in the case of $c^2 \ge T - \frac{1-e^{-4}}{4}$. 

\begin{theorem}
Let $c^2 \ge T - \frac{1-e^{-4}}{4}$, $T \le \frac{1}{4}$ and let curves $l_1 - l_4$ be defined as follows.

1. The curve $l_1$ is a segment of the boundary $\partial D(T)$ given by \eqref{Zher_eq12}, $Y \in [1-4T, Y_0]$, $Y_0$ is the unique solution of \eqref{Zher_eq14}.

2. The curve $l_2$ is given by solutions $(X,Y)$, $X+iY=z$, $\mu \in [0, 1]$ of the equation
\begin{equation}
\label{Zher_eq17}
z^2 + 1 - 2c(2\mu -1)(z - i) + 8\mu c^2(\mu - 1) \ln\frac{z+c(2\mu-1)}{i+c(2\mu-1)} = 4T.
\end{equation}

3. The curve $l_3$ is given by solutions (X,Y) of the system
\begin{equation}
\label{Zher_eq18}
\left\{
\begin{aligned}
&2p^2\log \frac{Y p}{c} + Y^2-p^2 = 1-4T-c^2,\\
&X = -c + p(1-\log\frac{Y p}{c}),
\end{aligned}
\right.
\end{equation}
where $p \in [c, p_0]$ and
\begin{equation}
\label{Zher_eq19}
p_0 = \sqrt{\frac{1}{2}(\sqrt{(4T+c^2-1)^2+4c^2} + (4T+c^2-1))}.
\end{equation}
The curve $l_4$ is symmetric to $l_3$ with respect to the imaginary axis.

If the following equation 
\begin{equation}
\label{Zher_eq20}
-4pc + \frac{c^2}{p^2}\exp{(-\frac{4c}{p})} - p^2 = 1- 4T - c^2
\end{equation}
has two solutions $p_1 < p_2$ in the interval $(c,p_0)$ we also define curves $l_5-l_{10}$.

4. The curve $l_5$ is given by solutions (X,Y) of the system \eqref{Zher_eq18}, $p\in[c,p_1]$. The curve $l_6$ is symmetric to $l_5$ with respect to the imaginary axis.

5. The curve $l_7$ is given by solutions (X,Y) of the system
\begin{equation}
\label{Zher_eq21}
\left\{
\begin{aligned}
&4cp +(X-c)^2-Y^2-4T = c^2 - 1,\\
&- p\log\frac{(X-c)Y}{c} = 2c,
\end{aligned}
\right.
\end{equation}
where $p \in [p_1, p_2]$. The curve $l_8$ is symmetric to $l_7$ with respect to the imaginary axis.

6. The curve $l_9$ is given by solutions (X,Y) of \eqref{Zher_eq18}, $p\in[p_2,p_0]$. The curve $l_{10}$ is symmetric to $l_9$ with respect to the imaginary axis.

The following two cases are possible: 

(1) $D_c(T)$ is bounded by curves $l_1,l_2, l_5-l_{10}$, if \eqref{Zher_eq20} has two solutions $p_1 < p_2$ in the interval $(c,p_0)$.

(2) $D_c(T)$ is bounded by curves $l_1-l_4$, if \eqref{Zher_eq20} has less than two solutions in the interval $(c,p_0)$.
\end{theorem}

\begin{figure}
\centering
    \begin{subfigure}[h]{0.495\textwidth}
        \epsfig{file=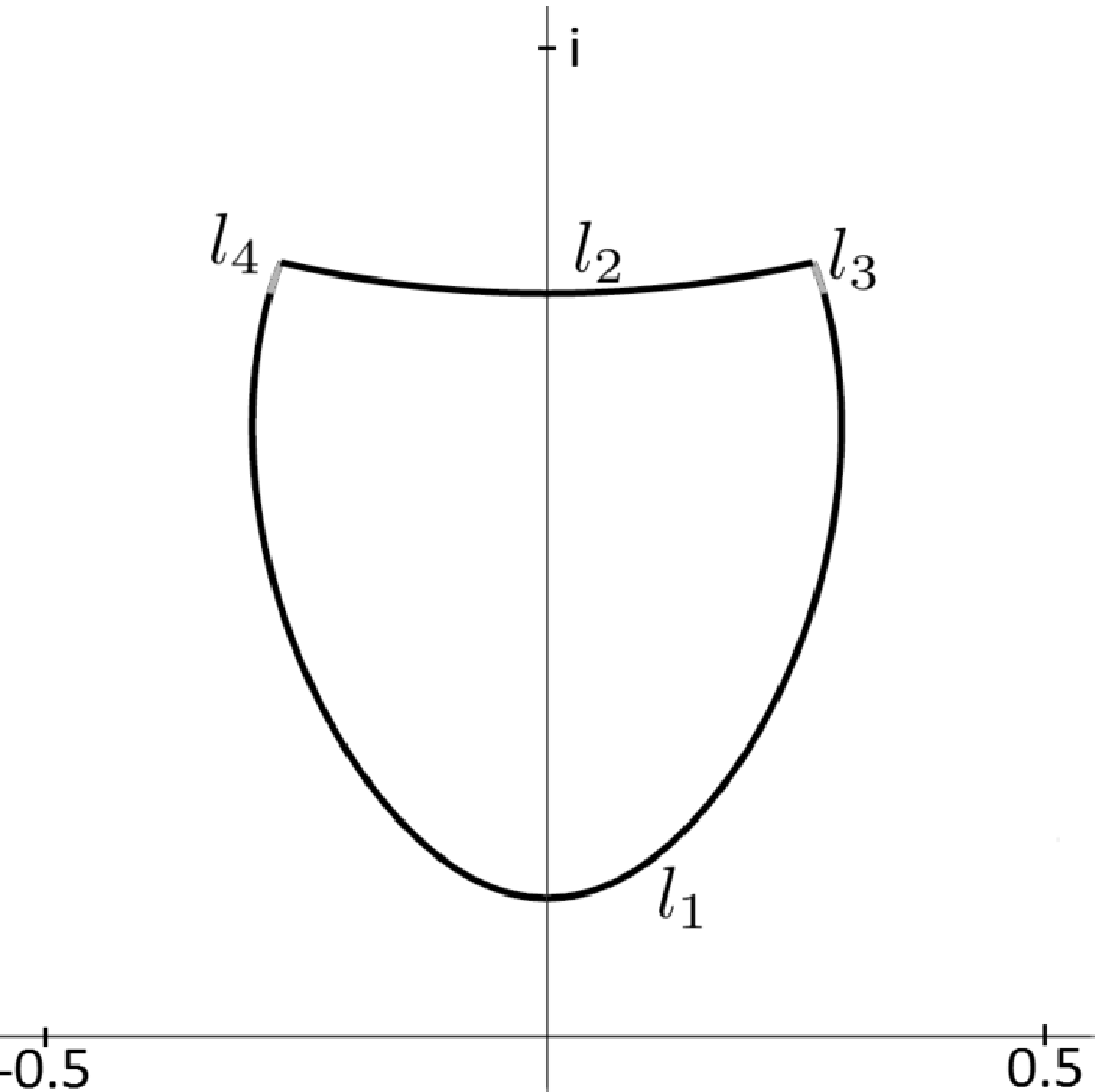, height=1.8in, width=1.78in}
        \caption{T=0.245, c=1}
    \end{subfigure}
    \centering
    \hspace{-10mm}
    \begin{subfigure}[h]{0.495\textwidth}
        \epsfig{file=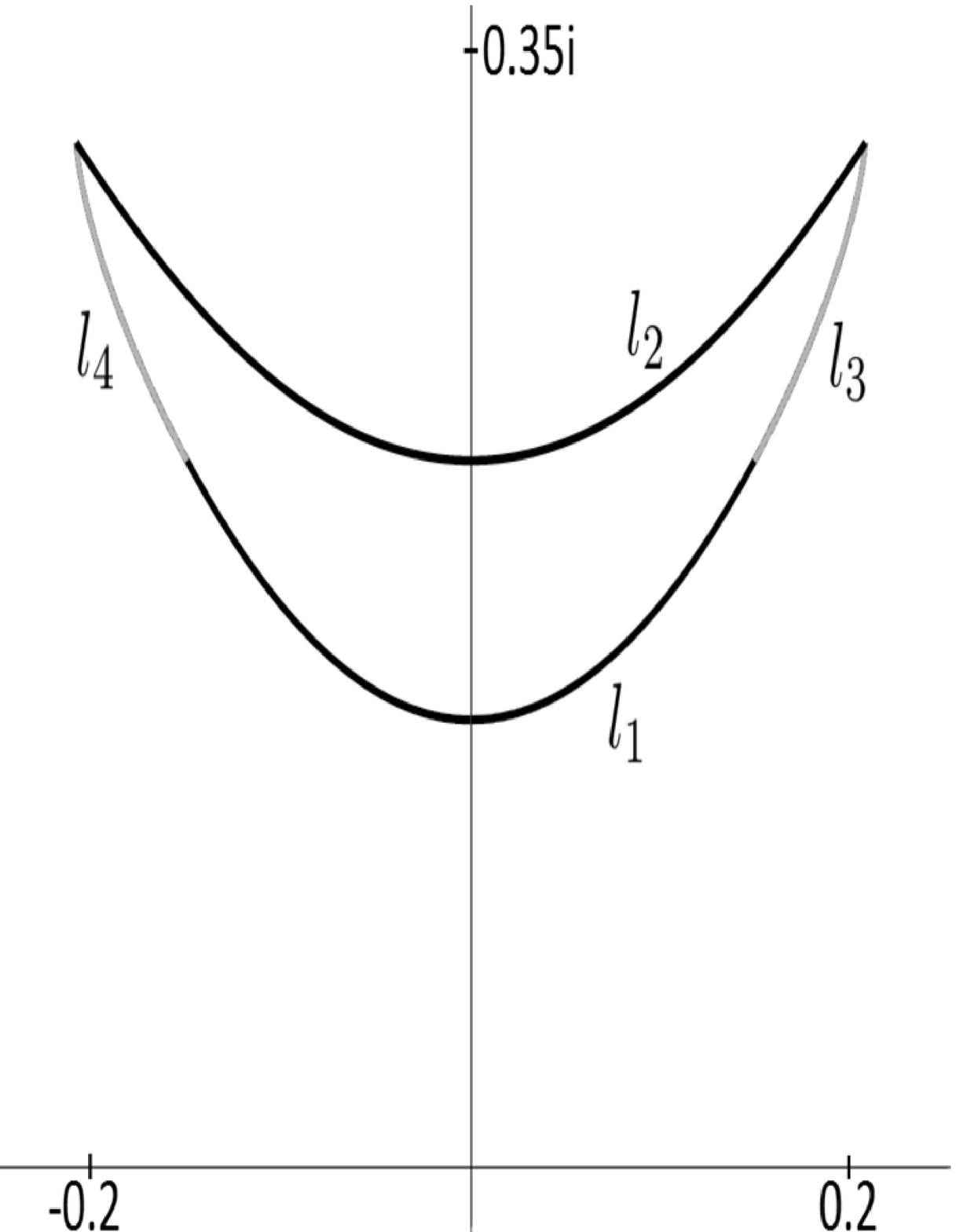, height=1.8in, width=2.3in}
        \caption{T=0.245, c=0.1}
    \end{subfigure}
\caption{The boundaries of the value ranges $D_c(T)$}
\label{Zher_fig2}
\end{figure}

\begin{proof}
The curve $l_1$ is already given by Lemma 2.
It can be seen from Lemma 1 that, at $t=0$, the Hamilton function $H$ is maximized at $\lambda_0 = -p$. Thus, if $p > c$, $H$ attains its maximum on $[-c, c]$ at $\lambda = -c$. Therefore, we have the following driving function:
\begin{equation}
\label{Zher_eq22}
\lambda(t) = \left\{
    \begin{aligned}
     &-c,\, &0 \le t \le t_1, \\
     &x(t)-p,\,  &t_1 < t \le T.
   \end{aligned}
   \right.
\end{equation}
Denote $x_1=x(t_1),\, y_1=y(t_1)$. Applying Lemma 3 to the interval $[0, t_1]$ we obtain
$$
(x_1 + c)y_1 = c, \, (x_1+c)^2 - y_1^2 - 4t_1 = c^2-1.
$$
Since $\lambda(t)$ is continuous, \eqref{Zher_eq22} gives $x_1 = p - c$. Thus, $y_1 = \frac{c}{p}$ and we can also find $t_1$:
\begin{equation}
\label{Zher_eq23}
4t_1 = p^2 - \frac{c^2}{p^2} - c^2 + 1.
\end{equation}
Integration of \eqref{Zher_eq6} and \eqref{Zher_eq7} over the interval $[t_1, T]$ yields the system \eqref{Zher_eq18}.
The equation \eqref{Zher_eq23} shows that $t_1$ increases as a function of $p$. Therefore, we can rewrite the condition $t_1 \in [0, T]$ as $p \in [c, p_0]$, where $c$ and $p_0$ are the roots of \eqref{Zher_eq23} for $t_1=0$ and $t_1=T$, respectively. Note that if $p=c$ equations \eqref{Zher_eq18} turn into \eqref{Zher_eq13}.

We have to satisfy the condition $\lambda(t) \in [-c, c]$. Since $\lambda(t)$ is equal to $-c$ on $[0, t_1]$ and increases on $[t_1, T]$, we only have to ensure
that $\lambda(T) \le c$. According to Lemma 2 for $p=c$ we have $\lambda(T) < c$. Assume that at some point $p \in (c, p_0]$, $\lambda(T) > c$. Then, due to continuity of $\lambda(T)$ as a function of $p$, there is a point $p_1 \in (c, p_0)$, such that $\lambda(T) = c$. Using \eqref{Zher_eq22} and \eqref{Zher_eq18} we can rewrite it as
$$
 -p\log\frac{Y p}{c} = 2c.
$$
With \eqref{Zher_eq18} it gives the equation \eqref{Zher_eq20} for $p_1$. Thus, if \eqref{Zher_eq20} has no roots in $(c, p_0)$ the case (2) takes place.  We note, however, the existing of a single root of \eqref{Zher_eq20} in $(c, p_0)$ does not guarantee a violation of the condition $\lambda(T) \le c$, and, thus, the case (2) is still possible.

Now consider the driving function
\begin{equation}
\label{Zher_eq24}
\lambda(t) = \left\{
    \begin{aligned}
     &-c,\, &0 \le t \le t_1, \\
     &x(t)-p,\,  &t_1 < t \le t_2,\\
     &c,\,  &t_2 < t \le T.
   \end{aligned}
   \right.
\end{equation}

\begin{figure}
\centering
        \epsfig{file=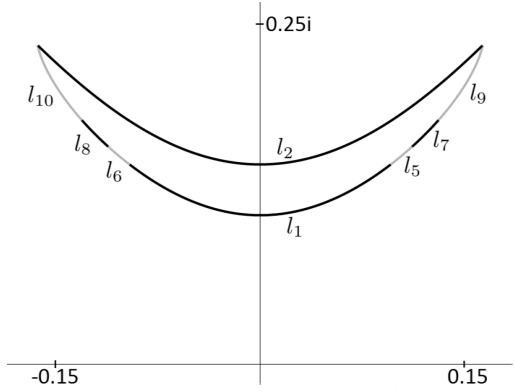, height=2in, width=2.64in}
\caption{The boundary of the value range $D_c(T), T=0.247, c=0.05$}
\label{Zher_fig3}
\end{figure}

Denote $x_2=x(t_2),\, y_2 = y(t_2)$. Applying Lemma 3 to the phase system for the interval $[t_2, T]$ we can write
\begin{equation}
\label{Zher_eq25}
(X-c)Y = (x_2-c)y_2,\, (X-C)^2 - Y^2 - 4T = (x_2-c)^2 - y_2^2 - 4t_2.
\end{equation}
The condition $\lambda(t_2) = c$ gives $x_2 = c + p$. Integrating \eqref{Zher_eq6} and \eqref{Zher_eq7} over the interval $[t_1, t_2]$ we obtain

$$
2p^2\log \frac{y_2 p}{c} + y_2^2-p^2 = 1-4t_2-c^2,\, x_2 = -c + p(1-\log\frac{y_2 p}{c}),
$$
that with \eqref{Zher_eq25} lead us to \eqref{Zher_eq21}.
These equations describe the boundary segment governed by the driving functions of the type \eqref{Zher_eq24}.

From \eqref{Zher_eq25} and \eqref{Zher_eq21} we can deduce the equation for $t_2$
\begin{equation}
\label{Zher_eq26}
4t_2 = 1 - c^2 + 4pc + p^2 - \frac{(X-c)^2 Y^2}{p^2},
\end{equation}
and, therefore, we have
$$
4(t_2-t_1) = 4pc + \frac{c^2 - (X-c)^2 Y^2}{p^2}.
$$
The second equation in \eqref{Zher_eq21} implies that $c^2 > (X-c)^2 Y^2$, therefore, the inequality $t_1 < t_2$ always holds. Since $t_1$ increases and takes the value $t_1=T$ at $p=p_0$, there is a point $p_2 \in [p_1, p_0]$, such that at this point $t_2 = T$. Substituting $t_2=T$ into \eqref{Zher_eq26} and using the second equation in \eqref{Zher_eq21} we again obtain the equation \eqref{Zher_eq20} for $p_2$. Thus, we see that existing of two roots of \eqref{Zher_eq20} $p_1 < p_2$ in the interval $[c, p_0]$ is a necessary condition for the case (1).

It is not difficult to see that the segment of the boundary corresponding to $p \in [p_2, p_0]$ is delivered by the driving functions of the type \eqref{Zher_eq22} and, consequently, is described by the system \eqref{Zher_eq18}.

For the remaining part of the boundary $\partial D_c(T)$ the Hamilton function is maximized outside of the interval $[-c, c]$ and, thus, we have $|\lambda(t)| = c$. Therefore, we can use the generalized Loewner equation (see \cite{Pro90, Pro93})
$$\frac{dg(z, t)}{dt} = \mu\frac{2}{g(z,t) - c} + (1- \mu)\frac{2}{g(z,t) + c},\, g(z,0) = z,\, \mu \in[0, 1].$$
Putting $z(t) = g(i, t)$ and integrating the equation over $[0, T]$ we obtain the equation \eqref{Zher_eq17} for the curve $l_2$, parameterized by $\mu\in[0, 1]$.
\end{proof}

\medskip
\noindent{\bf Acknowledgment.} This work was supported by the Russian Science Foundation (project no. 17-11-01229).


\address{
Saratov State University, 83 Astrakhanskaya str., Saratov, 410012, Russia\\
Petrozavodsk State University, 33 Lenin pr., Petrozavodsk, 185910, Russia}
\email{jerdevandrey@gmail.com}

\end{document}